\documentclass{elsarticle}
\usepackage{amssymb,amsmath, amsthm}

\newtheorem{theorem}{Theorem}[section]
\newtheorem{corollary}{Corollary}[section]
\newtheorem{lemma}{Lemma}[section]
\newtheorem{remark}{Remark}[section]

\newdefinition{definition}{Definition}[section]

\makeatletter

\newcommand{\R}{\mathbb{R}}
\newcommand{\N}{\mathbb{N}}

\let \al=\alpha
\let \be=\beta
\let \var=\varphi
\let \vare=\varepsilon

\let \th=\theta

\let \k=\kappa
\let \p=\partial
\let \med=\medskip

\let\ul=\underline
\let\ol=\overline

\def\q{\quad}

\numberwithin{equation}{section}

\begin{document}
\begin{frontmatter}

\title{A note on permanence of nonautonomous cooperative scalar population models  with delays}
\author{Teresa Faria\footnote{ 
 Fax:+351 21 795 4288; Tel: +351
21 790 4929.}}

\address{Departamento de Matem\'atica and CMAF,\\ Faculdade de Ci\^encias, Universidade de Lisboa,\\ Campo Grande, 1749-016 Lisboa, Portugal\\
tfaria@ptmat.fc.ul.pt}



\begin{abstract}  
For a large family of nonautonomous scalar-delayed differential equations used in population dynamics, 
some criteria for  permanence  are given,  as well as explicit upper and lower bounds for the asymptotic behavior of solutions. The method described here is based on comparative results with auxiliary monotone systems. In particular, it applies to   a  nonautonomous scalar model proposed as an alternative to the usual delayed logistic equation.
\end{abstract}

\begin{keyword}  
Delay differential equation; permanence; cooperative equation; quasimonotone condition.\\
{\it 2010 Mathematics Subject Classification}:  34K12, 34K25, 92D25
\end{keyword}

\end{frontmatter}

\section{Introduction} 
\setcounter{equation}{0}

In \cite{bb14}, Bastinec et al. studied the permanence of the following scalar nonautonomous delay differential equation (DDE) with a quadratic nonlinearity:
\begin{equation}\label{1.1}
\dot x(t)=\sum_{k=1}^m \al_k(t)x(t-\tau_k(t))-\be(t)x^2(t),\q t\ge 0,
\end{equation}
where $m$ is a positive integer, $\al_k,\be:[0,\infty)\to (0,\infty)$ are continuous, $\tau_k: [0,\infty)\to [0,\infty)$ are continuous and uniformly bounded, $0\le  \tau_k(t)\le \tau$ for some $\tau>0$, for $k=1,\dots,m,\, t\ge 0$. 

 In view of the biological interpretation of model \eqref{1.1}, only positive (or nonnegative) solutions of \eqref{1.1} are meaningful. In \cite{bb14}, the authors restrict their attention to solutions of \eqref{1.1} with initial conditions of the form
\begin{equation}\label{1.2}
x(\th)=\var(\th),\ -\tau\le \th\le 0,
\end{equation}
for $\var:[0,\infty)\to (0,\infty)$ continuous,  and added the contraint $\sum_{k=1}^m \tau_k(t)>0$ for all $t\ge 0$.
Using the positivity of the functions $\al_k(t),\be(t)$, it is easy to see that solutions of \eqref{1.1}-\eqref{1.2} are positive whenever they are defined. 

In a previous paper \cite{bb12}, the same authors  considered a simpler nonautonomous model
\begin{equation}\label{1.3}
\dot x(t)=r(t)\Big [\sum_{k=1}^m \al_kx(t-\tau_k(t))-\be x^2(t)\Big],\q t\ge 0,
\end{equation}
where the delay functions $\tau_k(t)$ satisfy all the conditions  above, $r(t)$ is continuous and satisfies $r(t)\ge r_0,\ t\ge 0,$ for some constant $r_0>0$, and $\al_k,\be$ are positive constants, $1\le k\le m$.

Model \eqref{1.1} is a  generalization of the DDE \eqref{1.3}, obtained by considering a more general form of nonautonomous coefficients.  The scalar DDE \eqref{1.3} has a positive equilibrium $K^*={1\over \be} \sum_{k=1}^m \al_k$, which  was proven in \cite{bb12} to be a global attractor of all its positive solutions without any further restriction. 
In general, \eqref{1.1} does not have a positive equilibrium, so criteria for either extinction -- when zero is a global attractor --  or persistence or permanence play a crucial role.   

Here, we set some standard notations. 
For \eqref{1.1} and for the DDEs hereafter,  $C:=C([-\tau,0];\R)$ ($\tau>0$) with the usual sup norm $\|\var\|_\infty=\sup_{\th\in[-\tau,0]}|\var(\th)|$ will be taken as the phase space. For an abstract DDE  in $C$,
\begin{equation}\label{1.4}
\dot x(t)=f(t,x_t),\q t\ge t_0,
\end{equation}
where $f:\Omega\subset \R\times C\to \R$ is continuous, $x_t$ denotes segments of solutions in $C$, $x_t(\th)=x(t+\th),-\tau\le \th\le 0$. If the solutions of initial value problems  are unique,  $x(t;t_0,\var)$ designates the solution  of  $\dot x(t)=f(t,x_t),x_{t_0}=\var$; we use simply $x(t;\var)$ for $x(t;0,\var)$. Even if it is not stated, we shall always assume that $f$ is smooth enough so that  initial value problems  associated with \eqref{1.4} have  unique solutions, with continuous dependence on data.
This is the case if $f(t,\var)$ is  uniformly Lipschitz continuous on the variable $\var\in C$ on each compact subset of $\Omega$. For $C^+:=\{ \var\in C:\var(\th)\ge 0$ for $-\tau\le \th\le 0\}$, initial conditions \eqref{1.2} are written  in the simpler form
$x_0=\var$ with $\var\in int\, (C^+).$
Cf. e.g.~\cite{kuang}, for  the concept of permanence given below, as well as for other standard definitions.

\begin{definition} 
The scalar DDE \eqref{1.4}
is said to be {\bf permanent} (in $S=int (C^+)$, or another $S\subset C^+\setminus \{0\}$) if there are positive constants $m_0,M_0$ with $m_0<M_0$ such that, given any $\var\in S$, there exists $t_*=t_*(\var)$  such that $m_0\le x(t,\var)\le M_0$ for $t\ge t_*$.
\end{definition}

A nice criterion for the permanence of \eqref{1.1} was established in \cite{bb14}, assuming only that the functions $\al_k(t),\be(t)$ are uniformly bounded from above and from below by positive constants.

 \begin{theorem}\label{thm1} \cite{bb14}
Assume that $\sum_{k=1}^m \tau_k(t)>0$ for all $t\ge 0$. If
\vskip 1mm
(h1) there are positive constants $\al_0,A_0,\be_0,B_0$ such that
$$\al_0\le \al_k(t)\le A_0,\q \be_0\le \be(t)\le B_0 \q {\rm for}\q t\ge 0,k=1,\dots,m,
$$
then the  solutions of the initial value problems \eqref{1.1}-\eqref{1.2} are positive and defined on $[0,\infty)$, and \eqref{1.1} is permanent in $int(C^+)$. Moreover,  for every solution $x(t)$ of \eqref{1.1}-\eqref{1.2} the estimates
\begin{equation}\label{1.5}
m_0\le \liminf_{t\to\infty} x(t)\le \limsup_{t\to\infty} x(t)\le M_0,
\end{equation}
hold with 
\begin{equation}\label{1.6}
m_0=\liminf_{t\to\infty}{1\over {\be(t)}} \sum_{k=1}^m\al_k(t),\q M_0=\limsup_{t\to\infty}{1\over {\be(t)}} \sum_{k=1}^m\al_k(t).
\end{equation}
\end{theorem}

\smallskip

The proof of this result in \cite{bb14} is  broken into several steps, and takes little advantage of the criterion established  previously by the authors in \cite{bb12}.
Here, we present an alternative proof  based on the fact that  equations \eqref{1.1} and \eqref{1.3}  satisfy the quasimonotone condition. In fact, we shall show later (cf.~Theorem 3.2) that we  need not  assume that $\sum_{k=1}^m \tau_k(t)>0$ for all $t\ge 0$, and that initial conditions may be taken in the larger set
$C_0:=\{ \var\in C^+:\var (0)>0\}.$ We recall that a scalar DDE \eqref{1.4}   satisfies the  {\bf quasimonotone condition} (on the cone $C^+$) if for any $t\ge t_0$ and $\var,\psi \in C^+$ with $\var\le \psi$ and $\var(0)=\psi (0)$, then $f(t,\var)\le f(t,\psi)$ (cf.~\cite{smith}, p.~78). Under this condition, the semiflow is monotone. If $d_\var f(t,\var)$ exists, is continuous on $[t_0,\infty)\times C^+$, and $d_\var f(t,\var)\psi\ge 0$ for $\var,\psi\in C^+$ and $\psi(0)=0$, then \eqref{1.4} is  {\bf cooperative}; cooperative equations satisfy the quasimonotone condition.
Here, we abuse the terminology, and refer to equations satisfying the quasimonotone condition as  {\it cooperative}.

\begin{proof}[Alternative proof of Theorem \ref{thm1}]
 Let $x(t)=x(t;\var)$ be the solution for an initial value problem \eqref{1.1}-\eqref{1.2}, defined on some maximal interval $[0,a)$ with $a\in (0,\infty]$. Then $x(t)$ satisfies the inequality
$\dot x(t)\le A_0\sum_{k=1}^m x(t-\tau_k(t))-\be_0x^2(t),\ t\ge 0.
$
Comparing with the  cooperative equation
\begin{equation}\label{1.7}
\dot u(t)= A_0\sum_{k=1}^m u(t-\tau_k(t))-\be_0u^2(t),\q t\ge 0,
\end{equation}
by Theorem 5.1.1 of \cite{smith} we have that $x(t)$ is bounded and defined on $[0,\infty)$, with $x(t)\le u(t),t\ge 0$, where
$u(t)$ is the solution of \eqref{1.7}-\eqref{1.2}. Moreover, by  \cite{bb12}  the equilibrium  $u^*={{mA_0}\over {\be_0}} $ is a global attractor for all positive solutions of \eqref{1.7}. We therefore conclude that
$ \limsup_{t\to\infty} x(t)\le u^*.$
In a similar way, we have
$\dot x(t)\ge \al_0\sum_{k=1}^m x(t-\tau_k(t))-B_0x^2(t),\ t\ge 0,
$
and by comparison with  the  cooperative equation
$$\dot v(t)= \al_0\sum_{k=1}^m v(t-\tau_k(t))-B_0v^2(t),\q t\ge 0,
$$
we obtain the lower bound $ \liminf_{t\to\infty} x(t)\ge v^*:={{m\al_0}\over {B_0}}.$ Hence, \eqref{1.1} is permanent. 

We now prove the estimates in \eqref{1.5}-\eqref{1.6}. Denote $\underline{x}= \liminf_{t\to\infty} x(t), \bar x= \limsup_{t\to\infty} x(t)$. By the fluctuation lemma, there is a sequence $(t_n)$, with $t_n\to\infty$ and $x(t_n)\to \bar x, \dot x(t_n)\to 0$. Fix  a small $\vare >0$, and take $T>0$  so that $x(t)\le \bar x+\vare$ for $t\ge T+\tau$. For $n$ large enough so that $t_n-\tau_k(t_n)\ge T$ and $x(t_n)\ge \bar x-\vare$, using \eqref{1.1} we get
\begin{equation*}
\begin{split}
\dot x(t_n)&=\be(t_n)\left [\sum_{k=1}^m {{\al_k(t_n)}\over {\be(t_n)}}x(t_n-\tau_k(t_n))-x^2(t_n)\right ]\\
&\le \be(t_n)\left [ {{\bar x+\vare}\over {\be(t_n)}} \sum_{k=1}^m \al_k(t_n)-(\bar x-\vare)^2\right ]\\
&\le \be(t_n)\bar x \left [ {1\over {\be(t_n)}} \sum_{k=1}^m \al_k(t_n)-\bar x\right ]+O(\vare).
\end{split}
\end{equation*}
By letting $n\to \infty$ and $\vare\to 0^+$, we obtain
$$0\le \limsup_n\be(t_n) \left [ {1\over {\be(t_n)}} \sum_{k=1}^m \al_k(t_n)-\bar x\right ],$$
which implies $M_0-\bar x\ge 0$, for $M_0$ as in \eqref{1.6}.
The other inequality in \eqref{1.5} is proven in a similar way. \end{proof}

The above method  used to prove Theorem \ref{thm1} motivated us to extend the same arguments to other scalar DDEs from population dynamics. The idea is to consider a broad class of {\it cooperative} differential equations with (possibly time-varying)  delays and nonautonomous coefficients, and use the theory of monotone dynamical systems to obtain its permanence by comparison with two auxiliary differential equations with constant coefficients, for which a globally attractive positive equilibrium exists. 

As a particularly important example,  we have in mind to apply this approach to establish the permanence of  the following scalar population model:
\begin{equation}\label{1.8}
\dot x(t)=\sum_{k=1}^m {{\al_k(t)x(t-\tau_k(t))}\over {1+\be_k(t)x(t-\tau_k(t))}}
-\mu(t)x(t)-\k(t)x^2(t),\q t\ge 0,
\end{equation}
where  $\al_k,\k:[0,\infty)\to (0,\infty), \be_k,\mu,\tau_k:[0,\infty)\to [0,\infty)$ are continuous and bounded, $1\le k\le m$.
We emphasize that  the study of the permanence  of \eqref{1.8} (with the additional constraints $\be_k(t)>0,\mu(t)>0, \sum_{k=1}^m \tau_k(t)>0$ on $[0,\infty)$)
was proposed in \cite{bb14} as a topic for further research. Obviously, Theorem \ref{thm1} applies only  to the very concrete model \eqref{1.1}, and therefore cannot be invoked to deal with \eqref{1.8}. For $N(t)=x(t)$, eq. \eqref{1.8} with $m=1$ reads as
$$N'(t)={{\al(t) N(t-\tau(t))}\over {1+\be(t)N(t-\tau(t))}}-\mu(t) N(t)-\k(t)N^2(t),
$$
which is (after a scaling)  the nonautonomous version of 
\begin{equation}\label{1.9}
N'(t)={{\gamma \mu N(t-\tau)}\over {\mu e^{\mu \tau}+k (e^{\mu \tau}-1)N(t-\tau)}}-\mu N(t)-\k N^2(t),
\end{equation}
where $\gamma,\mu,\k ,\tau >0$. Eq. \eqref{1.9} was derived by Arino et al.~[1] as an
 alternative formulation for the classical delayed logistic equation, also known as Wright's equation, given by 
 $\dot N(t)=rN(t)(1-N(t-\tau)/K),$  where $r$ denotes the intrinsic growth rate, $K$ is the carrying capacity, and $\tau$ the maturation delay. The coefficients in this  logistic equation are related to the ones in \eqref{1.9}   by $r=\gamma -\mu$ (for $\gamma, \mu$ the birth and mortality rates, respectively) and $K=(\gamma -\mu)/\k$.
 In Section 3, we shall study a class of scalar DDEs which includes \eqref{1.8} as a particular case. As another illustration of our technique, the permanence of a nonautonomous Nicholson's blowflies equation will also be studied.

\section{Auxiliary results on stability of equilibria}

In this section, we  address the global attractivity  of nonnegative equilibria for a family of nonautonomous cooperative scalar DDEs.
We start  with an auxiliary lemma from  \cite{faria}.

\begin{lemma}\label{lem21}  {} \cite{faria} Consider a scalar equation \eqref{1.4} in $C=C([-\tau,0];\R)$,
with $f:[t_0,\infty)\times C\to\R$  continuous, $t_0\in\R$, and let $S\subset C$ be an invariant set for the semiflow of
\eqref{1.4}. Suppose that $f$ satisfies the condition
\vskip 1mm
{(h2)} for $t\ge t_0$ and $\var\in S$ with $|\var(\th)|<|\var(0)|$, $\th\in[-\tau,0)$, then $\var(0)f(t,\var)<0.$ \vskip 1mm
\parindent =0cm Then, the solutions of \eqref{1.4} with initial conditions $x_0=\var\in S$ are defined and bounded on $[t_0,\infty)$
and $|x(t)|\le \|\var\|_\infty$ for $t\ge t_0$.
\end{lemma}

Consider now the family of scalar-delayed population models   given by
\begin{equation}\label{2.1}
\dot x(t)=\rho (t)\Big [R\big (x(t-\tau_1(t)),\dots,x(t-\tau_m(t))\big )-D(x(t))\Big],\q t\ge 0,
\end{equation}
where: $m\in\N$, $\tau_k:[0,\infty)\to\R$ are continuous and $0\le \tau_k(t)\le \tau$ for some positive constant $\tau,$ for $t\ge 0, 1\le k\le m$; $\rho:[0,\infty)\to (0,\infty),R:\R^m_+:=[0,\infty)^m\to [0,\infty), D:[0,\infty)\to [0,\infty)$ are continuous with  $R(0,\dots, 0)=D(0)=0$. Moreover, let $R,D$ be  smooth enough in order to ensure uniqueness of solutions. The condition $R(0,\dots, 0)=0$ is not essential for our analysis, but it corresponds to the general framework in  population dynamics models, since zero should be a steady state. We note that
 the particular case
$\dot x(t)= R(x(t-\tau))-D(x(t)),\ t\ge 0,$
was studied by Arino et al~[1].

 Write \eqref{2.1} as $\dot x(t)=f(t,x_t)$, and observe that  $f$ satisfies Smith's quasimonotone condition. As before, due to  biological reasons we are only interested in positive solutions. Rather than
 initial conditions in
$int(C^+),$ we shall consider $x_0=\var$, with $\var$ in the larger set  of admissible initial conditions
$$C_0=\{ \var\in C^+:\var (0)>0\}.$$
 It is clear that for $t\ge 0$ and $\var\in C^+$ with $\var (0)=0$, then $f(t,\var)\ge 0$. This implies that solutions of \eqref{2.1} with  initial conditions $x_0\in C^+$ are nonnegative \cite{smith}.  For $x(t)=x(t;\var), \var\in C_0$, $x(t)$ satisfies the ordinary differential inequality $\dot x(t)\ge -\rho(t)D(x(t))$, thus conditions $D(0)=0$ and  $x(0)=\var(0)>0$ yield $x(t)>0$ whenever it is defined. 
 In what follows, concepts as permanence and global asymptotic stability always refer to the solutions with  initial conditions $x_0=\var\in C_0$.

\med

In the sequel, the following assumptions will be considered:

\begin{itemize}

\item[(A1)]  $\rho(t)$ is bounded from below by a positive constant, i.e., $\rho(t)\ge \rho_0>0$ for all $t\ge 0$;
\item[(A2)]  $R(y_1,\dots,y_m)$ is nondecreasing in $y_k\in [0,\infty)$, $k=1,\dots,m$;
\item[(A3)] there exists $K\ge 0$ such that
\begin{equation}\label{2.2}(x-K) (R(x,\dots,x)-D(x) )<0\q {\rm for}\q x>0,x\ne K.
\end{equation}

\end{itemize}

Note that if 
there exists $K>0$ such that \eqref{2.2} holds, then \eqref{2.1} has the equilibria 0 and  $K$; whereas 0 is the unique equilibrium if \eqref{2.2} is satisfied with $K=0$.

Next result is a  generalization of Theorem 3.3 in  [1], and its proof can be found in the Appendix.
For related results, see \cite{faria} and Chapter 4 of Kuang's monograph \cite{kuang}.

\begin{theorem}\label{thm21}
 Consider equation \eqref{2.1}, and assume (A1)--(A3). Then the equilibrium $K$  in (A3) is globally asymptotically stable (GAS) in the set of  solutions with initial conditions in $C_0$.  
\end{theorem}

The same proof works with minimal changes for  equations with distributed delays, rather than  discrete delays, as stated below.

\begin{theorem}\label{thm22} Consider equation 
\begin{equation}\label{2.3}
\dot x(t)=\rho (t)\Big [R\big (L_1(t,x_t),\dots,L_m(t,x_t)\big)-D(x(t) )\Big ],\q t\ge 0,
\end{equation}
where $R,D$ are as in \eqref{2.1}, and $L_k:[0,\infty)\times C\to\R$ are defined by
$$L_k(t,\var)=\int_{-\tau}^0 \var (\th) \, d_\th\eta_k(t,\th),\q t\ge 0, \var\in C,$$
for some measurable functions $\eta_k:[0,\infty)\times[-\tau,0]\to\R$ such that $\eta_k(t,\cdot)$ is nondecreasing on $[-\tau,0]$ with $\eta_k(t,0)-\eta_k(t,-\tau)=1$ for all $t\ge 0$, $k=1,\dots,m$.
If (A1)--(A3) are satisfied, then the equilibrium $K$  in (A3) is GAS in the set of  solutions with initial conditions in $C_0$. 
\end{theorem}

For $L_k$ as in the above statement, $L_k(t,\cdot) :C\to\R$ are normalized, positive   linear operators for all $t\ge 0$. Clearly, eq. \eqref{2.1} is a particular case of \eqref{2.3}, where $\eta_k(t,\th)$ are Heaviside functions, $\eta_k(t,\th)=0$ for $-\tau\le \th \le -\tau_k(t), \eta_k(t,\th)=1$ for $-\tau_k(t)<\th\le 0$.

\med

We now consider a subclass of scalar DDEs \eqref{2.1}, since a significant number of population models have the form \eqref{2.1} (or \eqref{2.3}) with 
$R(y_1,\dots,y_m)= \sum_{k=1}^m y_kr_k(y_k)$, $D(x)=xd(x)$ for $ (y_1,\dots,y_m)\in \R^m_+,x\in \R_+,
$
 with $r_k,d$ strictly positive on $\R_+$. See  Theorem 3.3 in [1] for a particular case, and [5, pp.~146], also for further references. 

\begin{corollary}\label{cor21}
Consider the equation 
\begin{equation}\label{2.4}
\dot x(t)=\sum_{k=1}^m x(t-\tau_k(t))r_k\big (x(t-\tau_k(t))\big )-x(t)d(x(t)),\q t\ge 0,
\end{equation}
where $r_k,d:[0,\infty)\to [0,\infty)$ are locally Lipschitz functions with  $r_k(x)>0, d(x)>0$ for $x>0$, and $\tau_k:[0,\infty)\to [0,\tau]$ are continuous, $k=1,\dots,m$.
Assume also that
\begin{itemize}
\item[(i)]  $xr_k(x)$ are nondecreasing functions on $[0,\infty)$, $k=1,\dots,m$;
\item[(ii)] for $r(x):=\sum_{k=1}^m r_k(x)$, the function $r(x)-d(x)$ is   (strictly) decreasing  on $[0,\infty)$;
\item[(iii)] $r(\infty)-d(\infty)<0.$
\end{itemize}
Then, there are at most two nonnegative equilibria. If zero is the unique equilibrium, then it is GAS; if there is a positive equilibrium $x^*$, then $x^*$ is GAS (in the set of all solutions with initial conditions $x_0=\var\in C_0$).
\end{corollary}

\begin{proof} If $r(0)-d(0)\le 0$,  (A3) holds with $K=0$, hence 0 is GAS; if
$r(0)-d(0)> 0$, then there exists a unique $x^*>0$ such that $r(x^*)-d(x^*)=0$, and (A3) holds with $K=x^*$.\end{proof}

A similar version of  this corollary for equations with distributed delays could also be stated.

\section{ Main results}

Based on Theorem 2.1, we now extend the arguments used in our proof of Theorem \ref{thm1} to a larger class of nonautonomous cooperative models. 

\begin{theorem}\label{thm31} Consider
\begin{equation}\label{3.1}
\dot x(t)=R\big (t,x(t-\tau_1(t)),\dots,x(t-\tau_m(t))\big)-D(t,x(t)),\q t\ge 0,
\end{equation}
where $m\in\N$, $R(t,y), D(t,x),\tau_k(t)$ are continuous with
$0\le \tau_k(t)\le \tau$, for $t,x\in\R_+,y\in\R_+^m,1\le k\le m$, and assume that:
\begin{itemize}

\item[(H1)] there are (locally Lipschitz) continuous functions $R^l,R^u:\R^m_+\to\R_+,D^l,D^u:\R_+\to \R_+$ with $R^l(0,\dots, 0)=R^u(0,\dots, 0)=D^l(0)=D^u(0)=0$,
such that
$$R^l(y)\le R(t,y)\le R^u(y), D^l(x)\le D(t,x)\le D^u(x)\ {\rm for}\ t\ge 0, y\in \R_+^m, x\ge 0,
$$
and the pairs $(R^u,D^l), (R^l,D^u)$ satisfy assumptions (A2)-(A3) with $K>0$, i.e.,
\item[(H2)] $R^l(y_1,\dots,y_m),R^u(y_1,\dots,y_m)$ are nondecreasing functions on each $y_k\in [0,\infty)$, $k=1,\dots,m$;
\item[(H3)]   there exist $K^l,K^u>0$  such that
\begin{equation*}
\begin{split}
&(x-K^l) (R^l(x,\dots,x)-D^u(x) )<0\q {\rm for}\q x>0,x\ne K^l,\\
&(x-K^u) (R^u(x,\dots,x)-D^l(x) )<0\q {\rm for}\q x>0,x\ne K^u.
\end{split}
\end{equation*}
\end{itemize}
Then  \eqref{3.1} is permanent (in $C_0$); to be more precise,  all positive solutions of \eqref{3.1} satisfy
\begin{equation}\label{3.2}
K^l\le \liminf_{t\to\infty} x(t)\le \limsup_{t\to\infty} x(t)\le K^u.
\end{equation}
\end{theorem}

 \begin{proof}  As for the first part of the proof of Theorem \ref{thm1},   we compare the solutions $x(t)=x(t;\var)\ (\var\in C_0$) of  equation  \eqref{3.1} with the solutions $u(t)=u(t;\var)$ and  $v(t)=v(t;\var)$ of the cooperative equations
  \begin{equation*}
\begin{split}
\dot u(t)&= R^u\big (u(t-\tau_1(t)),\dots,u(t-\tau_m(t))\big )-D^l(u(t)),\\
\dot v(t)&=R^l\big (v(t-\tau_1(t)),\dots,v(t-\tau_m(t))\big )-D^u(v(t)),\q t\ge 0,
\end{split}
\end{equation*}
respectively.
By Theorem 5.1.1 in  \cite{smith} we deduce that $v(t)\le x(t)\le u(t)$ for $t\ge 0$, whereas Theorem 2.1 applied to these equations  implies that $v(t)\to K^l, u(t)\to K^u$ as $t\to\infty$. This implies \eqref{3.2}.\end{proof}

\begin{remark}  {\rm Of course, if in (H1)--(H3)  only  the conditions regarding $R^l$ and $D^u$ are assumed (and with no constraints on upper bounds for $R(t,y)$ and lower bounds for $D(t,x)$), instead of the permanence, only  the {uniform persistence} for \eqref{3.1}  is derived. Similarly, if   only  the conditions regarding $R^u$ and $D^l$ are assumed,  instead of the permanence,  one  simply gets that  \eqref{3.1} is dissipative. On the other hand, it is apparent that  more general models with distributed delays can be considered, in which case Theorem \ref{thm22} should be used for results of comparison with auxiliary cooperative systems.}
\end{remark}

The same technique and Corollary \ref{cor21} lead to the corollary below.

\begin{corollary}\label{cor31} Let $m\in\N$, $r_k(t,y), d(t,x),\tau_k(t)$ be continuous with
$0\le \tau_k(t)\le \tau$, for $t,x\ge 0,1\le k\le m$, and assume that:
\vskip 1mm
\item{(i)} there are (locally Lipschitz) continuous functions $r_k^l,r_k^u:\R_+\to\R_+,d^l,d^u:\R_+\to \R_+$ 
such that
$$r_k^l(x)\le r_k(t,x)\le r_k^u(x), \ d^l(x)\le d(t,x)\le d^u(x)\q {\rm for}\q t\ge 0, x\ge 0;
$$
\item{(ii)} $xr_k^l(x)$ and $xr_k^u(x)$ are nondecreasing functions on $[0,\infty)$, $1\le k\le m$; 
\item{(iii)} the functions $r^u(x)-d^l(x)$ and $r^l(x)-d^u(x)$, where $ r^u(x)=\sum_{k=1}^m r_kž(x),
r^l(x)=\sum_{k=1}^mr_k^l(x)$, are (strictly) decreasing on $[0,\infty)$;
\item{(iv)} $r^l(0)-d^u(0)>0$ and $\lim_{t\to\infty} (r^u(x)-d^l(x))<0.$
\vskip 1mm
Then,  the equation
\begin{equation*}
\dot x(t)=\sum_{k=1}^m x(t-\tau_k(t))r_k\big (t,x(t-\tau_k(t))\big )-x(t)d\big (t,x(t)\big ),\q t\ge 0,
\end{equation*}
 is permanent in $C_0$. Moreover,  \eqref{3.2} holds with $K^u,K^l$ the  positive solutions of the equations $r^u(x)-d^l(x)=0, r^l(x)-d^u(x)=0$, respectively.
\end{corollary}

\smallskip

We finally study the permanence of \eqref{1.8}, with less constraints than the ones proposed in \cite{bb14}.

\begin{theorem} \label{thm32} Consider  the equation 
\begin{equation}\label{3.3}
\dot x(t)=\sum_{k=1}^m {{\al_k(t)x(t-\tau_k(t))}\over {1+\be_k(t)x(t-\tau_k(t))}}
-\mu(t)x(t)-\k(t)x^2(t),\q t\ge 0,
\end{equation}
where 
$\al_k,\k:[0,\infty)\to (0,\infty)$ are continuous, bounded below and above by positive constants, and $\mu,\be_k,\tau_k:[0,\infty)\to [0,\infty)$ are continuous and bounded, for $1\le k\le m$. If
\begin{equation}\label{3.4}
\sum_{k=1}^m \inf_{t\ge 0}\al_k(t)>\sup_{t\ge 0}\mu(t),
\end{equation}
then all solutions $x(t)=x(t;\var)\ (\var\in C_0)$ of \eqref{3.3} satisfy the uniform estimates
\begin{equation}\label{3.5}
m_0\le \liminf_{t\to\infty} x(t)\le \limsup_{t\to\infty} x(t)\le M_0,
\end{equation}
where 
\begin{equation}\label{3.6}
M_0=\limsup_{t\to\infty}{1\over {\k(t)}} \left (\sum_{k=1}^m\al_k(t)-\mu(t)\right),
\end{equation}
 and
\begin{equation}\label{3.7}
m_0={c_0\over c_1}\q {\rm for}\q  c_0=\sum_{k=1}^m \inf_{t\ge 0}\al_k(t)-\sup_{t\ge 0}\mu(t),\ c_1=\sup_{t\ge 0}\k(t)+\sum_{k=1}^m \inf_{t\ge 0}\al_k(t)\, \sup_{t\ge 0}\be_k(t),
\end{equation}
For the particular case of \eqref{3.3} with $\be_k\equiv 0, 1\le k\le m$, i.e.,
\begin{equation}\label{3.8}
\dot x(t)=\sum_{k=1}^m \al_k(t)x(t-\tau_k(t))
-\mu(t)x(t)-\k(t)x^2(t),\q t\ge 0,
\end{equation}
where $\al_k,\mu,\k$ and $\tau_k$ are as above,
the lower bound in \eqref{3.5} can be taken as
\begin{equation}\label{3.9}
m_0=\liminf_{t\to\infty}{1\over {\k(t)}} \left (\sum_{k=1}^m\al_k(t)-\mu(t)\right).
\end{equation}
\end{theorem}

\begin{proof}    In what follows,   we set $0\le \tau_k(t)\le \tau$ for $t\ge 0,k=1,\dots,m$, and  use the notations
$$\underline f=\inf_{t\ge 0}f(t),\q \ol f=\sup_{t\ge 0}f(t),$$
for $f$ replaced by $\al_k,\be_k,\mu$ or $\k$. By 
 assumption, the functions $\al_k,\k$  are bounded and bounded away from zero, and $\be_k,\mu$ are positive and bounded.  Note that the cases $\ul\mu=0$ or $\ul{\be_k}=0$, for all or some $k$'s, are included in our setting.
For $R_k(t,x):={{\al_k(t)x}\over {1+\be_k(t)x}}$, we have  $\frac{\p }{\p x}R_k(t,x)>0$ for $t,x\ge 0, 1\le k\le m$; in particular \eqref{3.3} satisfies the quasimonotone condition.
 Next, denote  $$r^u(x)=\sum_{k=1}^mr_k^u(x),\ r^l(x)=\sum_{k=1}^mr_k^l(x),\q D^u(x)=xd^u(x),\ D^l(x)=xd^l(x)$$ where
$${r_k^u}(x)={ {\overline\al_k}\over {1+\underline{\be_k}x}},\q
{r_k^l}(x)={ {\ul\al_k}\over {1+\ol{\be_k}x}}\q {\rm for}\q x\ge 0,k=1,\dots,m,$$
and 
$d^u(x)=\ol \mu+\ol k x,\ d^l(x)=\ul \mu+\ul k x$ for $x\ge 0.$

The functions  $xr_k^u(x),xr_k^l(x)$ are increasing  and $r^u(x)-d^l (x),r^l(x)-d^u (x)$ are decreasing on $\R_+$.
On the other hand, $r^u(\infty)-d^l(\infty)=-\infty$ and \eqref{3.4} implies that $r^l(0)-d^u(0)>0$. By Corollary \ref{cor31}, \eqref{3.3} is permanent; furthermore, solutions $x(t)$ of  \eqref{3.3} with initial conditions in $C_0$ satisfy $K^l\le \ul x\le \ol x\le K^u$, for
$ \ul x:=\liminf_{t\to\infty} x(t), \ol x:=\limsup_{t\to\infty} x(t)$, where $K^l,K^u$ are the  globally attractive positive equilibria of
\begin{equation*}
\begin{split}
&\dot u(t)= \sum_{k=1}^m u(t-\tau_k(t)) r_k^u\big (u(t-\tau_k(t))\big )-u(t)d^l(u(t)),\\
&\dot v(t)= \sum_{k=1}^m v(t-\tau_k(t)) r_k^l\big (v(t-\tau_k(t))\big )-v(t)d^u(v(t)),
\end{split}
\end{equation*}
respectively.

To prove the uniform upper bound in \eqref{3.5}-\eqref{3.6},  we  reason along the lines of the proof of Theorem \ref{thm1}, so some details are omitted. Take a sequence $(t_n)$ with $t_n\to\infty$,
$\dot x(t_n)\to 0$ and $x(t_n)\to \ol x$. For any $\vare>0$ small and $n$ large, we derive
$$\dot x(t_n)\le k(t_n)\left [{1\over {k(t_n)}}\left (\sum_{k=1}^m R_k\big (t_n,\ol x+\vare\big )-\mu (t_n)x(t_n)\right )-x^2(t_n)\right ].
$$
Taking limits $n\to\infty,\vare \to 0^+$, this inequality yields
$$\ol x\le \limsup_{t\to\infty}{1\over {k(t)}} \left (\sum_{k=1}^m{{\al_k(t)}\over{1+\be_k(t)\ol x}}-\mu(t)\right)
\le M_0.$$

For the lower bound $m_0$ in \eqref{3.7}, we note that    $K^l$ is the unique $x>0$ such that $r^l(x)-d^u(x)=0$.
 Since $r^l(0)-d^u(0)=c_0$ and 
$|(r^l-d^u)'(x)|\le \sum_{k=1}^m \ul \al_k\ol \be_k+\ol\k=c_1$,
by the Mean Value theorem we get
$K^l\ge c_0/c_1$. 
In the case of \eqref{3.8}, we now take a sequence $(s_n)$ with $s_n\to\infty$,
$\dot x(s_n)\to 0$ and $x(s_n)\to \ul x$. For any $\vare>0$ small and $n$ large, we derive
$$\dot x(s_n)\ge k(t_n)\left [{1\over {k(s_n)}}\left ((\ul x-\vare)\sum_{k=1}^m \al_k (s_n)-\mu (s_n)x(s_n)\right )-x^2(s_n)\right ].
$$
Taking limits $n\to\infty,\vare \to 0^+$, this leads to $\ul x\ge m_0$ for $m_0$ as in \eqref{3.9}.
\end{proof}

\smallskip

\begin{remark} {\rm For eq.   \eqref{3.8} with $\mu\equiv 0$, we obtain  \eqref{1.1}. Thus, Theorem \ref{thm1} is simply a corollary of Theorems \ref{thm32}, the latter one under slightly weaker hypotheses.}
\end{remark}

Clearly, the method presented in this note applies to other scalar nonautonomous delayed population models, as illustrated in the next example.

Consider the scalar DDE
\begin{equation}\label{3.12}
\dot x(t)=-d(t)x(t)+\sum_{k=1}^m\be_k(t) x(t-\tau_k(t))e^{-x(t-\tau_k(t))}\, ,
\end{equation}
where $\be_k,\tau_k,d$ are continuous, bounded and nonnegative on $[0,\infty)$.
Eq. \eqref{3.12} is a generalization of the well-known Nicholson's equation $\dot x(t)=-dx(t)+\be x(t-\tau)e^{-x(t-\tau)} \ (d,\be,\tau>0)$. The autonomous version of \eqref{3.12} reads as
\begin{equation}\label{3.13}
\dot x(t)=-dx(t)+\sum_{k=1}^m\be_k x(t-\tau_k)e^{-x(t-\tau_k)}.
\end{equation}
With $d>0, \be_k\ge 0$ and $\be:=\sum_{k=1}^m\be_k>0$,
Liz et al. \cite{liz} proved that if $1< \be /d\le e^2$, then the positive equilibrium $x^*:=\log (\be /d)$ of \eqref{3.13} is  a global attractor of all  positive solutions. Although \eqref{3.13} is not cooperative, if $1< \be /d\le e$ then all solutions satisfy   $\lim_{t\to\infty} x(t)=x^*\le 1$, so  \eqref{3.13} has a cooperative large-time  behavior, since $h(x):=xe^{-x}$ is increasing on $[0,1]$. 
Under some  further constraints on the coefficients $\be_k(t),d(t)$, we shall take advantage of this monotonicity to study the permanence of \eqref{3.12}.

\begin{theorem} \label{thm33}
For \eqref{3.12}, suppose that 
\begin{equation}\label{3.14}
\sup_{t\ge 0} d(t)< \sum_{k=1}^m \inf_{t\ge 0} \be_k(t),\q  \sum_{k=1}^m \sup_{t\ge 0} \be_k(t)<e\, \inf_{t\ge 0} d(t).
\end{equation}
Then, all positive solutions $x(t)$ of  \eqref{3.12} satisfy
\begin{equation}\label{3.15}
\liminf_{t\to\infty} \log \left ({1\over {d(t)}}\sum_{k=1}^m\be_k(t)\right)\le \liminf_{t\to\infty} x(t),\ \limsup_{t\to\infty} x(t)\le  \limsup_{t\to\infty} \log \left ({1\over {d(t)}}\sum_{k=1}^m\be_k(t)\right).
\end{equation}
\end{theorem}

\begin{proof} With the notation
$\ul\be_k=\inf_{t\ge 0} \be_k(t),\ol\be_k=\sup_{t\ge 0} \be_k(t), \ul d=\inf_{t\ge 0} d(t), \ol d=\sup_{t\ge 0} d(t)$,  conditions \eqref{3.14} are written as
$\ol d< \sum_{k=1}^m \ul\be_k\le \sum_{k=1}^m \ol\be_k<e\ul d.$ Solutions of \eqref{3.12} satisfy
$$ -\ol d x(t)+\sum_{k=1}^m\ul \be_k h(x(t-\tau_k(t)))\le \dot x(t)\le -\ul d x(t)+\sum_{k=1}^m\ol \be_k h(x(t-\tau_k(t))),$$
where $h(x)=xe^{-x}$ for $x\ge 0$. Note also that $h(x)\le H(x)$, where $H$ is defined by  $H(x)=h(x)$ for $0\le x\le 1, H(x)=e^{-1}$ for $x>1$., and that the equation $\dot u(t)=-\ul d x(t)+\sum_{k=1}^m\ol \be_k H(x(t-\tau_k(t)))$ is cooperative.
From Corollary \ref{cor31}, \eqref{3.12} is permanent, and (3.2) holds with $m_0 =\log \big [(\sum_{k=1}^m \ul\be_k)/\ol d\big ]>0$ and $M_0=\log \big [(\sum_{k=1}^m \ol\be_k)/\ul d\big ]<1$.

Next, the better estimates in \eqref{3.15} are derived by using the technique in the proofs of Theorems \ref{thm1} and \ref{thm32}. For any positive solution of \eqref{3.12}, set $\ol x=\limsup_{t\to\infty} x(t), \ul x=\liminf_{t\to\infty} x(t)$. Take a sequence $t_n\to\infty$ with $\dot x(t_n)\to 0, x(t_n)\to \ol x$. Fix $\vare>0$. For $n$ large,
$$\dot x(t_n)\le -d(t_n)x(t_n)+h(\bar x+\vare)\sum_{k=1}^m\be_k(t)
= d(t_n)\bar x\Big [-1+{{e^{-\bar x}}\over {d(t_n)}}\sum_{k=1}^m\be_k(t_n) \Big ]+O(\vare).$$
Taking $n\to\infty$ and $\vare\to 0$, this yields
$\ol x\le \limsup_{t\to\infty} \log \left ({1\over {d(t)}}\sum_{k=1}^m\be_k(t)\right).$
A similar argument leads to
$ \ul x\ge \liminf_{t\to\infty} \log \left ({1\over {d(t)}}\sum_{k=1}^m\be_k(t)\right).$
\end{proof}

This example also shows an obvious limitation of  our method: since it relies on comparative results with cooperative equations,  it cannot be invoked to deal with \eqref{3.12} in the case of the upper bound in \eqref{3.14} given by $\sum_{k=1}^m \ol\be_k<e^2\ul d.$ On the contrary, an advantage of the method is that it can be easily extended to deal with $n$-dimensional cooperative DDEs.

\section*{Appendix}

For the sake of completeness, we include here the proof of Theorem 2.1, since the arguments in   [1], based on the theory of cooperative {\it autonomous} DDEs,  do not apply directly to equations of the form  \eqref{2.1} due to  the presence of time-varying delays.  See \cite{faria,kuang} for related results.
\med

\begin{proof}[Proof of Theorem \ref{thm21}] Eq. \eqref{2.1} has the form \eqref{1.4} with $t_0=0$ and 
$$f(t,\var)=\rho(t)\Big[R\big (\var(-\tau_1(t)),\dots, \var(-\tau_m(t))\big )-D(\var(0))\Big].$$
We now consider separately the cases $K=0$ and $K>0$ in (A3).

(i) If $K=0$, we set $S=C_0$ in Lemma \ref{lem21}.  Take $t\ge 0$ and $\var\in C$ with $0\le \var(\th)<\var(0)$ for $\th\in[-\tau,0)$. From (A2), (A3), 
$f(t,\var)\le \rho(t)[R(\var(0),\dots, \var(0))-D(\var(0))]<0.$
From Lemma \ref{lem21}, we deduce that all solutions  $x(t)=x(t;\var)$ of \eqref{2.1} are defined and bounded on $[0,\infty)$, and satisfy $0\le x(t)\le \|\var\|_\infty,t\ge 0.$ In particular, $x=0$ is a stable equilibrium. To prove that 0 is a global attractor, we need to prove that $\bar x:=\limsup_{t\to\infty} x(t)=0$ for any  nonnegative solution $x(t)$.

Arguing by contradiction, we suppose that $\bar x>0$, and take a sequence $t_n\to\infty$ such that $\dot x(t_n)\to 0$ and $x(t_n)\to \bar x$ as $n\to\infty$.  In order to simplify the notation, denote $R_0(x):=R(x,\dots,x)$, $x\ge 0$. Fix $\vare >0$. For $n$ sufficiently large, we have $x(t_n+\th)\le \bar x+\vare, \, 0\le \th\le 0$, thus
$$\dot x(t_n)\le  \rho(t_n)[R_0(\bar x+\vare)-D(x(t_n))].$$
Since $\rho(t_n)\ge \rho_0>0,\lim_n\dot x(t_n)=0, \lim_nD((x(t_n))=D(\bar x)$ and
$$R_0(\bar x+\vare)-D(\bar x)\to R_0(\bar x)-D(\bar x)<0\q {\rm as}\q \vare\to 0^+,$$
this leads to
$0\le \rho_0 [R_0(\bar x)-D(\bar x)]<0$,
which is a contradiction. Hence $\bar x=0$.

\med

(ii) If $K>0$, we effect the change of variables $y(t)={{x(t)}\over K}-1$, which transforms \eqref{2.1}  into
\begin{equation*}
\dot y(t)={{\rho(t)}\over K}\left [ R\big(K+Ky(t-\tau_1(t)),\dots, K+Ky(t-\tau_m(t))\big )-D(K+Ky(t))\right ],\ t\ge 0,\eqno{\rm (A.1)}
\end{equation*}
 for which $S=\{ \var \in C:\var(\th)\ge  -1$ for $\th\in [-\tau,0), \var(0)>-1\}$ is the set of admissible initial conditions. Using again \eqref{2.2} and  the fact that $R$ is nondecreasing on each variable, one sees that  (A.1) satisfies condition (h2). Thus, Lemma \ref{lem21} implies that all solutions  $x(t)=x(t;\var)$ of \eqref{2.1} are global and bounded, and that the equilibrium $x=K$ is  stable.

For solutions $y(t)=y(t;\var)$ of (A.1) (with $\var \in S$), define now
$-v=\liminf_{t\to\infty} y(t)$, $u=\limsup_{t\to\infty} y(t)$. We have $-1\le -v\le u<\infty$. It suffices to show that $\max (u,v)=0$. 

{\it Case 1}: If $\max (u,v)=u>0$, we take a sequence $t_n\to\infty$ such that $\dot y(t_n)\to 0$ and $y(t_n)\to u$ as $n\to\infty$. Reasoning as above, for any $\vare>0$ small and $n$ large, we obtain
$$
\dot y(t_n)\le  {{\rho(t_n)}\over K}\Big  [R_0(K(1+u+\vare))-D(K(1+y(t_n)))\Big ]$$
with $\rho(t_n)\ge \rho_0>0$,  and
$R_0(K(1+u+\vare))-D(K(1+y(t_n)))\to R_0(K(1+u))-D(K(1+u)) $
as $\vare\to 0^+$. Hence $0\le R_0(K(1+u))-D(K(1+u))$, which  is not possible with $u>0$.

\smallskip

{\it Case 2}: If  $\max (u,v)=v>0$, from the above case we may consider $u<v$.  We first prove that $v<1$. Choose $\vare>0$ small so that $u+\vare <v$, and $t_0$ such that $y(t)\le u+\vare$ for $t\ge t_0+\tau.$
By Lemma \ref{lem21}, the solution $y(t)$ satisfies
$$|y(t)|\le \max_{s\in [t_0-\tau,t_0]} |y(s)|\le \max \{u+\vare,\max_{s\in [t_0-\tau,t_0]} (-y(s))\}=:\ell,\q t\ge t_0.$$
In particular, $y(t)\ge -\ell>-1$ for $t\ge t_0$, thus $-v>-1$. We now choose a sequence $s_n\to \infty$ with  $\dot y(s_n)\to 0$ and $y(s_n)\to -v$ as $n\to\infty$. Proceeding along the lines of the above case, we now get
$0\ge R_0(K(1-v))-D(K(1-v))$, which contradicts \eqref{2.2}. The proof is complete.\end{proof}

\section*{Acknowledgement} The research was supported by Funda\c c\~ao para a Ci\^encia e a Tecnologia (Portugal),  PEst-OE/\-MAT/\-UI0209/2011.

\end{document}